\documentclass[a4paper,12pt]{article}

\usepackage[latin1]{inputenc}
\usepackage[T1]{fontenc}
\usepackage[english]{babel}

\usepackage{mathrsfs}
\usepackage{amscd}
\usepackage{amsfonts}
\usepackage{amsmath}
\usepackage{amssymb}
\usepackage{amstext}
\usepackage{amsthm}
\usepackage{amsbsy}

\usepackage{xspace}
\usepackage[all]{xy}
\usepackage{graphicx}
\usepackage{tikz}
\usepackage{url}
\usepackage{latexsym}

\usepackage{graphicx} 

\usepackage{booktabs} 
\usepackage{array} 
\usepackage{paralist} 
\usepackage{verbatim} 
\usepackage{subfig} 

\usepackage{fancyhdr} 
\usepackage{sectsty}
\allsectionsfont{\sffamily\mdseries\upshape} 

\usepackage[nottoc,notlof,notlot]{tocbibind} 
\usepackage[titles,subfigure]{tocloft} 

\usepackage[textwidth=100pt,textsize=footnotesize,bordercolor=white,color=blue!30]{todonotes}
\usepackage{hyperref} 

\makeatletter
\newcommand*{\rom}[1]{\expandafter\@slowromancap\romannumeral #1@}
\makeatother

\theoremstyle{definition}

\newtheorem{fact}{fact}

\newtheorem{thm}[fact]{Theorem}
\newtheorem{lemma}[fact]{Lemma}
\newtheorem{prop}[fact]{Proposition}
\newtheorem{corollary}[fact]{Corollary}
\newtheorem{defini}[fact]{Definition}

\title{Effectivity and Reducibility with Ordinal Turing Machines}
\author{Merlin Carl}
\date{}

\begin{document}


\maketitle

\begin{abstract}
This article expands our work in \cite{Ca16}.
By its reliance on Turing computability, the classical theory of effectivity, along with effective reducibility and Weihrauch reducibility, is only applicable to objects that are either countable or can be encoded by countable objects. We propose a notion
of effectivity based on Koepke's Ordinal Turing Machines (OTMs) that applies to arbitrary set-theoretical $\Pi_{2}$-statements, along with according variants of effective reducibility and Weihrauch reducibility. As a sample application, we compare
various choice principles with respect to effectivity. 

We also propose a generalization to set-theoretical formulas of arbitrary quantifier complexity.

\end{abstract}



\section{Introduction}

In mathematics, it has always been of interest to what extent theorems claiming the existence of certain objects are `effective' in the sense that a procedure - an algorithm - can be found that produces an object of the desired kind. In the case of a $\Pi_{2}$-statements, which, roughly, are claims that, for every $x$ from a certain range of objects, there is a $y$ from another such range such that $x$ and $y$ relate in a certain (simple) way, effectivity would then mean that some appropriate $y$ can be obtained effectively from a given $x$.

With the notion of Turing computability serving as a formalization of the notion of a `procedure', such questions can be turned into mathematical questions that potentially have provable answers; recursive analysis is an example of an area that deals with such questions (see e.g. \cite{Go}).

A large amount of standard theorems, however, turn out to be `ineffective' in this sense. For these theorems, it is still possible to compare their remoteness from effectivity by asking whether one $\Pi_{2}$-statement, say $\phi$, becomes effective once an effectivization for the other, say $\psi$, is given as an oracle. This yields the notion of effective reducibility. Insisting that only one instance of $\psi$ may be requested when processing an 
instance of $\phi$, one obtains the stronger notion of Weihrauch reducibility. All of these concepts have given rise to thriving areas of mathematics.

However, the reliance on Turing computability has the effect that the areas of mathematics thus surveyed are limited to countable objects or objects that, like separable spaces, can be encoded by countable objects. On the other hand, there seems to be an intuition of effectiveness even for general set theory; see e.g. \cite{Ho}. 

So far, the literature on effectivity notions for the uncountable is rather sparse; we mention on the one hand Hodges \cite{Ho}, which is the first such approach known to us, and further \cite{GHHM}, in particular the contribution by Shore \cite{Sh}. While this leads to effectiveness notions suitable for the uncountable, suitable analogues of effective reducibility and Weihrauch reducibility remain unconsidered, save for a remark in \cite{Sh}.

The first paper dealing with effective (Weihrauch) reducibility in the uncountable realm was \cite{Ca16}, where such notions were defined and some basic methods for their study were presented and applied. The guiding idea of \cite{Ca16} was to retain the usual definitions as far as possible while replacing Turing machines with Koepke's Ordinal Turing Machines (OTMs) (see e.g. \cite{Ko1}), that have a working tape of proper class length, the cells of which are indexed by ordinals, and additionally work along an ordinal `time axis'.  A related piece of work is Galeotti and Nobrega \cite{GN}, where a notion of Weihrauch reducibility based on OTMs was applied to various problems on generalized real numbers.

In this paper, we will expand the work started in \cite{Ca16} by (i) extending the realm of methods for obtaining results, (ii) considerably extending the number of applications by investigating how various set-theoretical choice principles compare in the sense of ordinal Weihrauch reducibility, (iii) considering connections between reducibility and provability and (iv) extending the notions of reducibility from \cite{Ca16} to arbitrary set-theoretical statements.

Many of the results explained here will also appear in the forthcoming \cite{Ca19}, and we want to express our gratitude both to the editors of Computability and the de Gruyter publishing company for the kind permission to use parts of this article in \cite{Ca19} and vice versa. Finally, this paper extends \cite{Ca16} in the sense that the contents of section $2$ and parts of the contents of $3$ and $4$ can be found there. In most cases, we again present the result, but refer to \cite{Ca16} for proofs.
%
%
%
%
%

\section{Setting the stage}

In this section, we will briefly recall the underlying model of computation, namely Koepke's Ordinal Turing Machines (OTMs) and then work towards appropriate notions of effectivity and reducibility for functions and problems on arbitrary sets.

\subsection{Ordinal Turing Machines}

Ordinal Turing Machines (OTMs) were introduced by Koepke in \cite{Ko1} as a "symmetric" variant of the Infinite Time Turing Machines of Hamkins and Kidder that were introduced in \cite{HL}. The same concept was obtained almost at the same time independently by Dawson and is contained in his thesis \cite{Da}. 

Programs for OTMs are just regular Turing programs, where we imagine the inner states to be indexed by natural numbers. Also, like classical Turing machines, OTMs have a finite number of tapes, split into cells, each of which can contain one element of the alphabet $\{0,1\}$; moreover, as usual, there is a read-write-head for each tape that moves along the cells, reads out the symbol on the current cell and can replace it when instructed to do so. However, both the "hardware" and the working time of OTMs are considerably expanded: The tape of an OTM has proper class length On, with one cell for each ordinal, and programs are carried out along an ordinal "time axis". At successor times, an OTM behaves like a classical Turing machine, with the additional rule that, if the head is commanded to go to the left while currently occupying a cell indexed by a limit ordinal, it is reset to the start of the tape, i.e. the cell indexed $0$. At limit times, for each cell $c$, the content of $c$ is the inferior limit of the sequence of earlier contents of $c$, and likewise, the index of the active program line and the head position are obtained as the inferior limits of the sequence of earlier values. For an explication of OTMs in formal detail, we refer the reader to \cite{Ko1}. In this paper, we will only consider parameter-free OTMs, i.e. we always start with an empty working tape.

In our framework, OTMs work with four tapes, one input tape $T_{\text{in}}$, one working tape $T$, one output tape $T_{\text{out}}$ and one "miracle" tape $T_{\text{m}}$ that will be explained below. Occasionally, we will also use a fifth tape, the "oracle tape" $T_{\text{o}}$, for storing extra information. An input for an OTM is always a set or class $X$ of ordinals, which is given to the OTM by writing a $1$ on the $\iota$th cell of $T_{\text{in}}$ if and only if $\iota\in X$, and otherwise a $0$.\footnote{Sometimes, when $X$ is a set, it will be important to know where it `stops', i.e. to know some upper bound for $X$, e.g. to be able to "exhaustively search" through $X$. This can clearly be arranged in a number of ways (e.g. by having two input tapes, one for inputs and one for bounds), but in order to simplify our presentation, we will not discuss this point any further.} Thus, OTMs can compute on sets or classes of ordinals. In order to make OTMs compute on sets, we need to fix some encoding of sets by sets of ordinals, which we do as follows: For a transitive set $x$, pick some bijection $f:\alpha\rightarrow x$ between $x$ and some ordinal, and then let $c_{f}(x):=\{p(\iota,\zeta):\iota,\zeta\in\alpha\wedge f(\iota)\in f(\zeta)\}$. Then $c_{f}(x)$ is an "$\alpha$-code for $x$ via $f$"; when $c$ is an $\alpha$-code for $x$ via $f$ for some $\alpha$ and $f$, we say that $c$ "codes" $x$. Now, for arbitrary sets $x$, a code for $x$ is a code for $\{x\}\cup\text{tc}(x)$, where $\text{tc}(x)$ denotes the transitive closure of $x$.

Thus, OTMs can be used to compute on arbitrary sets, and we can say that a (class) function $F:V^{n}\rightarrow V$ (where $V$ is the set-theoretical universe) is OTM-computable if and only if there is an OTM-program $P$ such that, for each $(x_{1},...,x_{n})\in V^{n}$ and each set $c$ of ordinals that codes $(x_{1},...,x_{n})$, $P$ will, when run with $c$ on the input tape and all other tapes empty, eventually halt with some $d$ on the output tape such that $d$ codes $F(x_{1},...,x_{n})$. 
Note that we allow the specific code written to the output tape to depend on the code chosen for the input. If $X\subseteq V$ is such that, for any OTM-computable $F:V\rightarrow V$ and any $x\in X$, we have $F(x)\in X$, then we call $X$ `closed under OTM-computability' or `OTM-closed'.
Similarly, a class $X\subseteq\text{On}$ is OTM-decidable if and only if there is an OTM-program that, on input $c$ coding $x$, halts with output $1$ if and only if $x\in X$ and otherwise halts with output $0$.

We recall a few results about OTMs that will be relevant below.

\begin{thm}{\label{parameter-freely computable}} 
(i) [Seyfferth, Schlicht, \cite{SeSc}] A set $x$ is coded by an OTM-computable set if and only if $x\in L_{\sigma}$, where
$\sigma$ is minimal such that $L_{\sigma}\prec_{\Sigma_{1}}L$.

(ii) [Koepke, see \cite{Ko1}] There is a (non-halting) OTM-program $P_{L}$ that "enumerates $L$", i.e. that successively writes codes for all constructible sets (and no others) to the output tape.

(iii) [Koepke, \cite{Ko1}] There is an OTM-program $P_{\text{truth}}$ such that, for any $\Delta_{0}$-statement $\phi$ and all sets $x_{1},...,x_{n}$, $P_{\text{truth}}(\lfloor\phi\rfloor,c(x_{1}),...,c(x_{n}))$ halts with the truth value of $\phi(x_{1},...,x_{n})$ on the output tape. Here, $\lfloor\phi\rfloor$ is the G\"odel number of $\phi$ and $c(x_{1}),...,c(x_{n})$ are codes for the sets $x_{1},...,x_{n}$.
\end{thm}

\subsection{Effectivity of $\Pi_{2}$-statements}

It has become customary to associate the term "effective" with "recursive" in mathematical contexts. However, as e.g. Hodges has pointed out in \cite{Ho}, there is apparently an intuition about a notion of effectiveness that goes beyond Turing computability, according to which e.g. forming quotient fields of integral domains is effective, while obtaining maximal ideals of rings is not (cf. Hodges, \cite{Ho}, p. 1).
Hodges proposes a variant of the primitive recursive set functions of Jensen and Karp \cite{JK} as a formalization of this notion. Partly motivated by our argument in \cite{Ca} that there are good reasons to accept OTMs as a formal counterpart to the intuitive notion of "higher effectivity", we base our approach on OTMs instead.\footnote{Though not formally equivalent, the two approaches agree with respect to the problems considered by Hodges.} Another reason is that the OTM-approach seems more suited to considering reducibility between problems, which is not considered by Hodges. A slight change in perspective is that the emphasis in Hodges' paper is on "construction problems", while ours is on statements. However, for the transition from "problems" to "statements", it suffices to note that $\Pi_{2}$-statements are naturally associated with "problems":

Let $\phi$ be a $\Pi_{2}$-sentence in the language of set theory, say $\phi$ is $\forall{x}\exists{y}\psi(x,y)$, where $\psi$ is $\Delta_{0}$. We say that $\phi$ is "witnessed" by the proper class function $F:V\rightarrow V$ if and only if $\psi(x,F(x))$ holds for all $x\in V$. In this case, we also call $F$ a "canonification" of $\phi$.

More generally, when $R\subseteq V\times V$ is a binary relation on sets, we say that $F$ is a canonification of $R$ if and only if, for every $x\in V$, we have $R(x,F(x))$. 

\begin{defini}{\label{eff defini}}
A binary relation $R\subseteq V\times V$ is `OTM-effective' if and only if there is an OTM-computable canonification $F$ for $R$. 
Fixing the notation $R_{\psi}:=\{(x,y):\psi(x,y)\}$ for a $\Delta_{0}$-formula $\psi$, we say that the $\Pi_{2}$-statement $\phi\equiv\forall{x}\exists\psi(x,y)$ with $\psi\in\Delta_{0}$ is OTM-effective if and only if $R_{\psi}$ has an OTM-computable canonification $F$; in this case, we also say that $F$ is a canonification for $\phi$.

We say that an OTM-program $P$ is "functional" if and only if $P$ computes a class function $F:V\rightarrow V$. When $P$ is functional, we denote the function computed by it by $F_{P}$.
\end{defini}

In fact, this definition of effectivity can be generalized to arbitrary sentences in the language of set theory by (roughly) replacing Turing machines with OTMs in Kleene's realizability interpretation intuitionistic logic. The above definition then becomes a special case of this more general approach and one can show that the axioms of KP are OTM-effective, while the "impredicative" ZFC-axioms, like the power set axiom or separation and replacement for formulas with arbitrary use of unbounded quantifiers, are not. We will briefly touch on this topic in section $6$ below and further refer the interested reader to the forthcoming \cite{Ca19} for more information.

We fix some further terminology. For $F:V\rightarrow V$, $X\subseteq V$, we say that $X$ is $F$-closed if and only if $F(x)\in X$ for all $x\in X$. 
For $R\subseteq V^{2}$, $X$ is $R$-closed if and only if it is $F$-closed for some canonification $F$ of $R$. Moreover, for $R\subseteq V^{2}$, $X\subseteq V$, we let $R[X]:=\{y:\exists{x\in X}(x,y)\in R\}$ and $R^{-1}[X]:=\{x:\exists{y\in X}(x,y)\in R\}$.

Clearly, every OTM-effective $\Pi_{2}$-statement $\phi$ holds in $V$. If we also require $\phi$ to be "provably OTM-effective", i.e. that the existence of the desired program is provable in ZFC, then $\phi$ will be provable in ZFC. We will see below that not every $\Pi_{2}$-consequence of ZFC must be OTM-effective. However, if the universe is small, this is indeed the case:

\begin{prop}{\label{constructibly trivial}}
If $V=L$, then every $\Pi_{2}$-consequence of ZFC is effective.
\end{prop}
\begin{proof}
Let $\phi\equiv\forall{x}\exists{y}\psi(x,y)$ be a $\Pi_{2}$-consequence of ZFC, where $\psi$ is $\Delta_{0}$.
Given a code $c$ for a set $x$ on the input tape, run $P_{L}$ from Theorem \ref{parameter-freely computable}. 
Whenever $P_{L}$ outputs a code $c^{\prime}$ for a set $y$, use $P_{\text{truth}}$ to check whether $\psi(x,y)$ holds. If not, continue with $P_{L}$. Otherwise, output $c^{\prime}$ and halt. Clearly, this will compute a canonification of $\phi$.
\end{proof}

\subsection{Reducibility}

Though studying the effectivity, or otherwise, of various mathematical statements is certainly an extensive and worthwhile task, we will mostly be concerned with a more general notion, namely the "reducibility" of one statement or problem to another. In order to formalize this, we need to be able to use class functions $F:V\rightarrow V$ 
as "oracles". We will write $P^{F}$ to indicate that the OTM-program $P$ is run in the oracle $F$. This is where the "miracle tape" $T_{\text{m}}$ comes in. We reserve a special inner state $s$ as our "miracle command". Now, whenever the state $s$ is assumed during a run of $P^{F}$, the following takes place: The content of $T_{\text{m}}$ is read out; if it does not code a set, nothing further happens and the program execution is continued. Otherwise, if it codes the set $x$, then the content of $T_{\text{m}}$ is replaced with a code for $F(x)$ and then the execution of the program is continued. Thus, whenever the state $s$ is assumed, a code for the $F$-image of the set coded by the content of the miracle tape appears on the miracle tape.

Now we can define our reducibility relations:

\begin{defini}{\label{reducibility}}
Let $R,R^{\prime}\subseteq V^{2}$. We say that $R$ is "OTM-reducible" to $R^{\prime}$ if and only if the following holds: There is an OTM-program $P$ such that, for every canonification $F$ of $R^{\prime}$, $P^{F}$ computes a canonification of $R$. In this case, we write $R\leq_{\text{OTM}}R^{\prime}$.
\end{defini}

In this setting, arbitrarily many instances of $R^{\prime}$ may be used in computing a canonification of $R$. We may also demand that each instance of $R$ is reduced to a single instance of $R^{\prime}$. This yields the following reducibility notion, adapted from the classical notion of Weihrauch reducibility (see e.g. \cite{BGM}) which is the one we will be primarily interested in in this article:

\begin{defini}{\label{ordinal weihrauch reducibility}}
Let $R,R^{\prime}\subseteq V^{2}$. We say that $R$ is "ordinal Weihrauch reducible" to $R^{\prime}$ if and only if there are functional OTM-programs $P$ and $Q$ such that, for every canonification $F^{\prime}$ of $R^{\prime}$, $F_{P}\circ (F^{\prime},\text{id})\circ F_{Q}$ is a canonification of $R$ (where $\text{id}:V\rightarrow V$ denotes the identity function).


In this case, we write $R\leq_{\text{oW}}R^{\prime}$ and say that $[P,Q]$ witnesses the ordinal Weihrauch reducibility of $R$ to $R^{\prime}$.\footnote{In \cite{Ca16}, the notation was $\leq_{\text{gW}}$ for `generalized Weihrauch reducibility'. D. Dzhafarov pointed out in his review of \cite{Ca16} that this name is already taken and suggested `ordinal Weihrauch reducible' instead, a suggestion which we thankfully follow here.} The corresponding concepts and notations for $\Pi_{2}$-statements are defined in the obvious way.

If 
the above holds with $F_{P}\circ F^{\prime}\circ F_{Q}$ being a canonification of $R$ (i.e. with the identity function deleted in the middle term),
we say that $R$ is `strongly ordinal Weihrauch reducible' to $R^{\prime}$ and write $R\leq_{\text{soW}}R^{\prime}$.

For each of these reducibility relations, we write $R\equiv_{i}R^{\prime}$ with the respective index $i$ to express that $R\leq_{i}R^{\prime}$ and $R^{\prime}\leq_{i}R$.
\end{defini}

We note the obvious fact that strong ordinal Weihrauch reducibility implies ordinal Weihrauch reducibility, which in turn implies OTM-reducibility. We will see below that the second implication cannot be reversed. For the first one, this is easy to see: For example, every OTM-effective function $F:V\rightarrow V$ is $\leq_{\text{oW}}$-reducible to any relation $R\subseteq V\times V$, while, for instance, when $\text{id}$ and $0$ denotes the identitity and the constant zero function on $V$, respectively, then $\text{id}\nleq_{\text{sOW}}0$, as the image of $F\circ 0\circ G$ will be $\subseteq\{0\}$ for any choice of $F$ and $G$, while the image of $\text{id}$ is clearly class-sized. This idea is well-known from the classical setting; see Proposition \ref{otm slim} below for a general version. Though $\leq_{\text{oW}}$ seems to be the more natural reducibility relation, most reducibility results we will prove below will in fact be strong reducibilities, a phenomenon also well-known from the classical setting. (Our nonreducibility results, however, will mostly apply to $\leq_{\text{oW}}$.)

Furthermore, we note that all three reducibility relations are reflexive, transitive and antisymmetric, i.e. partial orderings, and the the corresponding $\equiv$-relations are thus equivalence relations.

%
%

%
%
%
%

\section{Basic methods}

In this section, we will recall the methods for proving non-reducibility that we adapted in \cite{Ca16} from the methods used by Hodges in \cite{Ho} for proofs of non-effectivity and add some more possibilities adapted from the classical theory of Weihrauch degrees, see \cite{BGM}.

Our first method for proving nonreducibility is adapted from Hodges' "cardinality method", see Lemma 3.2 of \cite{Ho}.
First, let us fix some terminology (cf. Hodges, \cite{Ho}) : For $\alpha\in\text{On}$, we say that a class function $F:V\rightarrow V$ `raises cardinalities above $\alpha$' if and only if there is some set $x$ with $|\text{tc}(x)|\geq\alpha$ such that $|F(x)|>|\text{tc}(x)|$. For $R\subseteq V^{2}$, we say that $F$ raises cardinalities above $\alpha$ if and only if every canonification of $R$ raises cardinalities above $\alpha$. 

\begin{lemma}{\label{cardinality method}}[Cf. \cite{Ho}, Lemma 3.2]
Let $R,R^{\prime}\subseteq V^{2}$, and let $\alpha$ be an infinite ordinal. Suppose that $R$ raises cardinalities above $\alpha$ and $R\leq_{\text{OTM}}R^{\prime}$. Then $R^{\prime}$ raises cardinalities above $\alpha$. Thus, a relation that raises cardinalities above $\alpha$ is not OTM-reducible - and hence also not ordinal Weihrauch reducible - to one that does not.

In particular, if $R$ raises cardinalities above $\omega$, then $R$ is not effective.
\end{lemma}
\begin{proof}
See \cite{Ca16}, Lemma $1$. The main ingredient is the well-known observation that OTM-computations cannot output sets of larger cardinalities than their input when the latter is infinite.
\end{proof}

We note some sample applications:

\begin{corollary}{\label{applications cardinality}}[Cf. \cite{Ca16}, Lemma $2$]
None of the following relations is effective:

\begin{itemize}
\item The relation between an ordinal and its cardinal successor.
\item The relation between an ordinal and its cardinal successor in $L$.
\item The relation between a set and its (constructible) power set.
\item The relation between a linear ordering and its completions.
\end{itemize}
\end{corollary}

For the separation of $\Pi_{2}$-statements, however, the cardinality method is useless, as binary $\Pi_{2}$-relations never raise cardinalities for any infinite $\alpha$ by an easy collapsing argument. 

Our second method is adapted from Hodges' "Forcing Method", see Hodges \cite{Ho}, Lemma $3.7$. The idea is that, when $M\subseteq V$ is OTM-closed and $R^{\prime}$-closed, then $R\leq_{\text{oW}}R^{\prime}$ implies that it is also $R$-closed. Thus, to show that $R\nleq_{\text{oW}}R^{\prime}$, it suffices to find some OTM-closed set or class $M$ with some $x\in M$ such there is no $y\in M$ with $(x,y)\in R$. 

If $X$ satisfies the well-ordering principle and thus contains codes for all of its elements, then in order to guarantee OTM-closure, it suffices that $M$ is a union of transitive KP-models that contain all ordinals (or is even only of sufficient ordinal height).

Below, we will be particularly interested in separating various versions of the axiom of choice. Thus, $M$ will in general not satisfy the well-ordering principle. The problem is that, in this case, OTM-closure is not obvious: As $M$ will not contain encodings for all of its elements, it will also not contain the corresponding computations and is thus not guaranteed to contain the set coded by the output. The idea then is that, for $x\in M$, there should be two mutually generic filters $G_{0}, G_{1}$ for some forcing that makes $x$ well-ordered. $M[G_{0}]$ and $M[G_{1}]$ will contain codes for $x$ and also all halting OTM-computations on these codes, along with their outputs and the sets they code. For an OTM-program $P$ that computes a class function $F_{P}:V\rightarrow V$, we will thus have $F(x)\in M[G_{0}]\cap M[G_{1}]$, and if $M$ is a model of sufficiently much set theory to guarantee that $M[G_{0}]\cap M[G_{1}]=M$, it follows that $F(x)\in M$, as desired.

Thus, it is necessary to guarantee the existence of the required filters. This requires some extra assumption beyond ZFC. 
Indeed, as noted above (see Proposition \ref{constructibly trivial}, some largeness assumption is necessary, for under $V=L$, all $\Pi_{2}$-consequences of ZFC are (trivially) ordinal Weihrauch reducible to each other. Though much less is actually required, a convenient extra assumption is $0^{\sharp}$, as it implies that every forcing with a parameter-free definition in $L$ is countable and thus has a generic filter, see e.g. p. 59 of Friedman \cite{Fr}, thereby covering all cases we are interested in below. We thus chose to work under $0^{\sharp}$.

For a set $x$, we denote by $\mathbb{P}_{x}$ the set $\{f:\omega\rightarrow\{x\}\cup\text{tc}(x)|f\text{ finite}\}$, partially ordered by $\subseteq$, which forces $x$ to be well-ordered in order type $\omega$.

\begin{thm}{\label{forcing method}}[Cf. \cite{Ca16}, Thm. 8]
Let $R,R^{\prime}\subseteq V^{2}$, and let $M$ be a transitive increasing union of transitive, class-sized KP-models such that $M\cap\text{On}=\text{On}$. Suppose that there are a canonification $F^{\prime}$ of $R^{\prime}$ and a set $x$ such that $M$ is $F^{\prime}$-closed, but does not contain a $y$ with $(x,y)\in R$. Moreover, suppose that there are two filters $G_{0}$ and $G_{1}$ that are mutually $\mathbb{P}_{x}$-generic over $M$. Then $R\nleq_{\text{oW}}R^{\prime}$.
\end{thm}
\begin{proof}
See \cite{Ca16}. We sketch the proof for the convenience of the reader: The condition on $G_{0}$ and $G_{1}$ implies that $x$ has codes in $M[G_{0}]$ and $M[G_{1}]$; also, the forcing will preserve enough of KP so that all halting OTM-computations starting on these codes will be contained in both extensions. So both extensions will contain the coded set, which will thus lie in the intersection. Thus $(F_{P}\circ F^{\prime}\circ F_{Q})(x)$ will be contained in $M$ for all functional OTM-programs $P$ and $Q$, while by assumption, no $y$ with $(x,y)\in R$ is contained in $M$. Hence, there are no $P$ and $Q$ that witness the ordinal Weihrauch reducibility of $R$ to $R^{\prime}$, as desired.
\end{proof}

\bigskip
\textbf{Remark}: The condition that $M$ contains all ordinals is unnecessarily strong: It is only needed to ensure that the generic extensions are "long" enough so that all OTM-computations that do in fact halt halt there. Further, the condition "increasing union of transitive, class-sized KP-models" can be further weakened e.g. to $M$ being a model of Mathias' provident set theory (see \cite{Ma}). So far not found applications for these variants.

\subsection{Further Methods}

We mention a few other ways for proving nonreducibility. In general, many of the methods used in the classical theory of Weihrauch reducibility depend so little on the specific model of computation that they immediately carry over to our context; we give two examples of this observation, namely Propositions \ref{comp criterion} and \ref{otm slim}. The intuition between the following two propositions is that, when $R$ is ordinal Weihrauch reducible to $R^{\prime}$ and there are subclasses $X$ and $Y$ of $V$ that are closed under OTM-computability such that, for all $x\in X$, there is $y\in Y$ with $(x,y)\in R^{\prime}$, then the same holds for $R$; more informally, when `simple' instances of $R^{\prime}$ have `simple solutions', then the same holds for $R$.


\begin{prop}{\label{comp criterion}} [For the classical version, cf. Brattka e.a., \cite{BGM}], Proposition $12.1$]
When $A$ is a set or a class of ordinals, a set $x$ is called OTM-computable in $A$ if and only if there is an OTM-program $P$ such that, when $P$ is run 
with $A$ on the oracle tape, it halts with a code for $x$ on the output tape.

Now let $A$ and $B$ be classes of ordinals, and let $R,R^{\prime}\subseteq V^{2}$. Suppose that, whenever $x$ is OTM-computable in $A$, then there is some $y$ 
that is OTM-computable in $B$ such that $(x,y)\in R^{\prime}$ and that $R\leq_{\text{oW}}R^{\prime}$. Then it also holds that, whenever $x$ is OTM-computable in $A$, then there is some $y$ that is OTM-computable in $B$ such that $(x,y)\in R$. 

In particular, when there is a canonification of $R^{\prime}$ that sends each OTM-computable set to an OTM-computable set and $R\leq_{\text{oW}}R^{\prime}$, then there is such a canonification for $R$ as well.
\end{prop}
\begin{proof}
The proof is the same as that for the classical version in \cite{BGM}: Assume that the assumptions of the proposition are satisfied, let $[P,Q]$ witness the reducibility of $R$ to $R^{\prime}$, and let $x$ be OTM-computable in $A$. Let $\hat{x}$ be the set coded by the output of $Q$ when run on a code for $x$ as its input.
Then $\hat{x}$ is also OTM-computable in $A$. Thus, by assumption, there is a set $\hat{y}$ that is OTM-computable in $B$ such that $(\hat{x},\hat{y})\in R^{\prime}$.
Similarly, if $y$ is the set coded by the output of $P$ when run on a code for $\hat{y}$ as its input, then $y$ is OTM-computable in $B$, as $y$ is.
As $[P,Q]$ witnesses the ordinal Weihrauch reducibility of $R$ to $R^{\prime}$, we have $(x,y)\in R$. Thus $y$ is as desired.
\end{proof}

%

For $c\subseteq\text{On}$, $\sigma^{c}$ denotes the supremum of the halting times of OTMs that start with $c$ on the oracle tape. 
For a set $x$, we let $\sigma^{x}:=\text{min}\{\sigma^{c}:c\text{ codes }x\}$.

\begin{prop}{\label{sigma criterion}}
Let $R,R^{\prime}\subseteq V^{2}$ such that $R\leq_{\text{oW}}R^{\prime}$, and let $X\subseteq V$ be OTM-closed and $\alpha\in\text{On}$. Suppose that, for every $x\in X$, there is $y$ such that $(x,y)\in R^{\prime}$ and $\sigma^{y}<\alpha$. Then it also holds that, for every $x\in X$, there is $y\in Y$ such that $(x,y)\in R$ and $\sigma^{y}<\alpha$.

In particular, if any OTM-computable instance $x$ of $R^{\prime}$ has a solution $y$ with $\sigma^{y}<\alpha$, then the same is true of $R$.
\end{prop}
\begin{proof}
Let $[P,Q]$ witness the reduction of $R$ to R$^{\prime}$ and let $x\in X$ be an instance of $R$. 
Let $x^{\prime}$ be the set coded by the output $c^{\prime}$ of $Q(c)$ when $c$ codes $x$. 
Then $\sigma^{c^{\prime}}\leq\sigma^{c}<\alpha$ by transitivity of computations. By assumption on $R^{\prime}$, there is $y^{\prime}$ with a code $d^{\prime}$ such that $(x^{\prime},y^{\prime})\in R^{\prime}$ and $\sigma^{d}<\alpha$. If $d$ is the output of $P$ run on the input $d^{\prime}$, then $\sigma^{d}<\alpha$ as above. Consequently, if $y$ is the set coded by $d$, then $\sigma^{y}<\alpha$ and $(x,y)\in R$, so $y$ is as desired.
\end{proof}


Another smooth generalization of a classically important observation to our setting is the following:

\begin{prop}{\label{otm slim}} [For the classical version, cf. Brattka et al., \cite{BGM}, section $13$]
Let $R,R^{\prime}\subseteq V^{2}$, $R\leq_{\text{soW}}R^{\prime}$, $X\subseteq V$ OTM-closed. Suppose that $R$ is `slim on $X$', i.e. for every $z\in R[X]$, there is some $x\in X$ such that $z$ is unique with $(x,z)\in R$. Then $|R[X]|\leq|R^{\prime}[X]|$.
\end{prop}
\begin{proof}
The proof of Proposition 13.1 in \cite{BGM} carries over almost verbatim.
%
%
\end{proof}

%

\section{An Application: Choice Principles}

We will now apply the methods obtained above to a comparison between various set-theoretical choice principles with respect to ordinal (Weihrauch) reducibility. 

Specifically, we will deal with the following statements, further variants of which will be introduced below; all of these statements are $\Pi_{2}$ and thus naturally associated with a "construction problem".

\begin{itemize}
\item "Every non-empty set contains an element" (PP, "picking principle")
\item "Every non-empty finite set contains an element" (PP$_{\text{fin}}$)
\item "Every $2$-element set contains an element" (PP$_{2}$) 
\item "Every non-empty set has a non-empty finite subset" (MPP, "multiple picking principle")
\item "For every family $X$ of pairwise disjoint, non-empty sets, there is a set $x$ that has finite non-empty intersection with any element of $X$" (MuC, principle of multiple choice)
\item "Every family of pairwise disjoint, non-empty sets has a system of representatives" (AC, axiom of choice)
\item "Every family of non-empty sets has a choice function" (AC$^{\prime}$)
\item "Every vector space has a basis" (VC)
\item "Every set can be well-ordered" (WO, well-ordering principle)
\item "Every partially ordered set has a maximal chain" (HMP, Hausdorff's maximality principle)
\item "Every partially ordered set in which every chain has an upper bound contains a maximal element" (ZL, Zorn's Lemma)
\end{itemize}

We begin with some obvious observations:

\begin{prop}{\label{WO is maximal}}
\begin{enumerate}[i]
\item PP$_{2}\leq_{\text{soW}}$PP$_{\text{fin}}\leq_{\text{soW}}$PP$\equiv_{\text{soW}}$ZL$\leq_{\text{soW}}$AC.
\item MuC$\leq_{\text{soW}}$AC$\equiv_{\text{soW}}$AC$^{\prime}$.
\item When $\phi$ is any of the above principles, then $\phi\leq_{\text{soW}}$WO.
\item MuC$\leq_{\text{soW}}$VC
\end{enumerate}
\end{prop}
\begin{proof}
(i) The first three reducibilities work via $[P,P]$, where $P$ computes the identity function. For PP$\leq_{\text{soW}}$ZL, given a non-empty set $x$ 
to pick from, pass over $(x,\text{id})$ to ZL, i.e. the partial ordering on $x$ where all elements are maximal; then any canonification of ZL will pick out an element 
of $x$. For the other direction, given a partially ordered set $(x,\leq)$ satisfying the condition of ZL, ZL implies that it will contain some maximal element.
Given $(x,\leq)$, it is easy to compute the set $M\subseteq x$ of maximal elements. Applying PP to $M$ then yields a maximal element of $x$. For reducing PP (or ZL) to AC,
 hand over $\{x\}$ to a canonification of AC when an instance $x$ of PP is given as the input.

(ii) MuC$\leq_{\text{soW}}$AC is trivial. AC$\equiv_{\text{soW}}$AC$^{\prime}$ works via the obvious OTM-effectivization of the equivalence proof over ZF, see e.g. Devlin \cite{De}, section 2.7.

(iii) Again, the usual proofs of implication over ZF effectivize on OTMs. For the picking principles, one well-orders the set $x$ and then picks the minimal element.
For VC, given a well-ordered vector space $V$, it is easy to see that the subset of $V$ consisting of those elements that are not linearly generated by those preceding it in the well-ordering forms a basis of $V$ (see e.g. Shore, \cite{Sh}) and thus OTM-computable from such a well-ordering.

(iv) This works by an easy effectivization of the proof in Blass \cite{Bl}.
\end{proof}

\bigskip
\textbf{Remark}: The reduction of ZL to PP is somewhat unsatisfying. We hence propose to consider HMP instead, which is the combinatorial "core" behind ZL.

\bigskip

When arbitrarily many instances of a problem can be used, the hierarchy collapses:


\begin{prop}{\label{choice collapse}}
We have PP$\equiv_{\text{OTM}}$WO. Thus, all of the above principles except PP$_{2}$ and PP$_{\text{fin}}$ are $\leq_{\text{OTM}}$-equivalent to WO.
\end{prop}
\begin{proof}
Again, the usual proof effectivizes. Let $P$ be an OTM-program that works as follows: Given a (code of a) set $x$, iteratively use PP to pick elements from the remaining set and exclude them from the set; eventually, every element of $x$ has been picked, and the order in which they were picked is a well-ordering of $x$.
\end{proof}

Though PP may seem too obvious to be called a "principle", note that this is not so in terms of effectivity: There is no general prescription for picking an element from an arbitrary set! Indeed, not even the simplest version is effective (recall that we are working under $0^{\sharp}$):

\begin{prop}{\label{pp2 noneffective}}
PP$_{2}$ is not effective.
\end{prop}
\begin{proof}
Suppose otherwise, and let $P$ be an OTM-program that computes a canonification $F$ of PP$_{2}$. By Jech \cite{JAC}, section 5.4, there is a symmetric extension $M$ of $L$ in which ZF holds while AC$_{2}$, the axiom of choice for countable sequences of pairs, fails. Let $s:=(\{a_{i},b_{i}\}:i\in\omega)\in M$ be such a sequence without a choice function in $M$. We can assume that the $a_{i}$ and $b_{i}$ are countable sets of real numbers in $M$. 

By $0^{\sharp}$, there are two mutually generic filters over $M$ that make $\text{tc}(s)$ well-ordered. Then both generic extensions, say $N$ and $N^{\prime}$, will contain codes for each $\{a_{i},b_{i}\}$, so $P$ can be applied to each such pair in both models, and thus, $F|s$ is definable over $N$ and $N^{\prime}$. 
Thus $F[s]$ will be a choice function for $s$ contained both in $N$ and $N^{\prime}$ by replacement, and so $F[s]\in N\cap N^{\prime}=M$, a contradiction.

\end{proof}

%
%
%
%
%
%

A similar idea allows us to separate versions of the picking principle:

\begin{thm}
 PP$\nleq_{\text{oW}}$PP$_{2}$. Thus PP$_{2}<_{\text{oW}}$PP. 
\end{thm}
\begin{proof}
We write AC$_{2}$ for the axiom of choice for sets of pairs.

Now suppose otherwise, and let $[P,Q]$ witness the reduction of PP to PP$_{2}$. 
By Theorem 7.11 of \cite{JAC}, there is a transitive set model of ZF+AC$_{2}$ with a countable set without a choice function.
Under the assumption of $0^{\sharp}$, the same construction works over the ground model $L$ (as the required filters exist) and yields
a transitive class model $M$ with the same properties.

Let $X\in M$ be a countable set without a choice function in $M$. For each code $c_{x}$ of an element $x\in X$, 
let $\tilde{x}$ be the set coded by the output of the computation $Q(c_{x})$.

By the assumption of $0^{\sharp}$, the forcing for making $\text{tc}(x)$ well-ordered has two mutually generic filters 
$G_{0},G_{1}$ over $M$. Thus, all elements of $x$ have codes both in $M[G_{0}]$ and $M[G_{1}]$, and therefore, the corresponding
$Q$-computations exist in $M[G_{0}]$ and $M[G_{2}]$. Thus, for each $x\in X$, we have 
$\tilde{x}\in M[G_{0}]\cap M[G_{1}]$ and, by replacement in $M[G_{0}]$ and $M[G_{1}]$, in fact 
$\tilde{X}:=\{\tilde{x}:x\in X\}\in M[G_{0}]\cap M[G_{1}]=M$. 
Now, since a canonification of PP$_{2}$ is applicable to $\tilde{x}$ for each $x\in X$, $\tilde{X}$ is a set 
of pairs. By AC$_{2}$ in $M$, $\tilde{X}$ has a choice function $f\in M$. Let $F\supseteq f$ be a canonification 
of PP$_{2}$ that extends $f$. Thus, $F_{P}\circ F\circ F_{Q}$ is a canonification of PP. 

Let $Y:=F[\tilde{X}]=f[\tilde{X}]\in M$. By the same argument as above, the collection $Z$ of sets 
coded by the outputs of $P$ when applied to an element of $Y$ is contained in two mutually generic extensions of $M$
that make $\text{tc}(Y)$ well-ordered and thus we have $Z\in M$. 

But now, by the assumptions on $P$, $Q$ and $F$, $Z$ is a choice function for $X$ contained in $M$, a contradiction.

\end{proof}

\bigskip
\noindent
\textbf{Remark}: The same approach works to turn many independence results over ZF between versions of AC into non-reducibility results between versions of PP. Thus, for example, the results given in section 7.4 of Jech \cite{JAC} imply
in the same way that PP$_{\text{fin}}<_{\text{oW}}$PP, that PP$_{2}$ and PP$_{3}$ are $\leq_{\text{OTM}}$-incomparable, that MPP$<_{\text{oW}}$PP etc.
Moreover, it in fact suffices to construct such (transitive, class-sized) models of KP (and possibly less), which is enough to guarantee the existence of the relevant computations.

\bigskip
On the other hand, by an obvious OTM-effectivization of the proof of the implication AC$_{2}\rightarrow$AC$_{4}$ over ZF
in \cite{JAC}, Example 7.12\footnote{For the sake of the reader, here is the idea of the proof in \cite{JAC}: Suppose that $F$ is a canonification of PP$_{2}$ and let $X$
be a set with four elements. Apply $F$ to all $6$ pairs of elements of $X$. Each element of $X$ will be thus picked a certain number of times (possibly $0$) as a representative of a pair. Let us consider those that get picked 
the smallest number of times. As $6$ is not divisible by $4$, there are one, two or three of those. If there is only one, pick it. If there are three, pick the remaining one. If there are two, use $F$ to pick one of them.}, we also get:

\begin{prop}
PP$_{4}\leq_{\text{OTM}}$PP$_{2}$ 
\end{prop}

\bigskip
\textbf{Question}: Is PP$_{4}\leq_{\text{oW}}$PP$_2$? In general, let us write PP$_{k,n}$ for the problem of picking $k$ elements from a set with $n$ elements.
Then how do the PP$_{k,n}$ relate with respect to $\leq_{\text{oW}}$ and $\leq_{\text{soW}}$?

\begin{prop}{\label{sicac}}
(i) PP$<_{\text{oW}}$AC. 

(ii) PP$<_{\text{oW}}$HMP

(ii) MPP$<_{\text{oW}}$MuC.
\end{prop}
\begin{proof}
Clearly, PP reduces to AC, since, given a non-empty set $x$, an element of $x$ is obtained by applying AC to $\{x\}$. So  PP$\leq_{\text{oW}}$AC

On the other hand, 
let $M$ be a transitive class model of ZF+$\neg$AC. Being transitive, $M$ contains an element of $x$ for every non-empty $x\in M$.
Let $F:M\rightarrow M$ be a class function such that $F(x)\in x$ for every $\emptyset\neq x\in M$. Let $\hat{F}:V\rightarrow V$ be a class function extending $F$
such that $\hat{F}(x)\in x$ for every $\emptyset\neq x\in V$. Now $M$ is closed under $\hat{F}$, which is a canonification of PP, but contains
some $x$ contradicting AC. By Lemma \ref{forcing method}, PP$\nleq_{\text{oW}}$AC. 


The arguments for (ii) and (iii) are completely analogous.
\end{proof}

On the other hand, PP is surprisingly strong:

\begin{prop}{\label{wezf}}
Suppose that $\phi$ is a set-theoretical $\Pi_{2}$-statement such that ZF$\vdash\phi$. Then $\phi\leq_{\text{oW}}$PP. In other words, PP is "$\Pi_{2}$-universal for ZF".
\end{prop}
\begin{proof}
Let $a$ be an arbitrary set and write $\phi$ in the form $\forall{x}\exists{y}\psi$ with $\psi$ a $\Delta_{0}$-formula. 
Let us write $\bar{x}:=\text{tc}(a\cup\{a\})$.

We show how to compute a (code for a) set $Y$ such that $\psi(a,y)$ holds for all $y\in Y$. This is achieved by the OTM-program $P$ that, given a code for $a$, enumerates $L[\text{tc}(\bar{a}]$ 
until the first $L[\bar{a}]$-level $L_{\beta}(\bar{a})$ is found that contains some $y$ such that $\psi(a,y)$. Whether or not $\psi(a,y)$ holds is easily OTM-decidable, since $\psi$ is $\Delta_{0}$.

Then the output is $Y:=\{y\in L_{\beta}(\bar{a}):\psi(a,y)\}$. As $L[\bar{a}]$ is a model of ZF, such a level must eventually be found, so the program always terminates with a non-empty output $Y$ that has
the desired properties.

Now applying PP to $Y$ yields some $b$ such that $\psi(a,b)$. Thus, we have reduced $\phi$ to PP.
\end{proof}

We have an analogous statement to Proposition \ref{wezf} for ZFC and WO in place of ZF and PP: Namely, WO is, with respect to set-theoretical $\Pi_{2}$-statements provable in ZFC, maximal for oW-reducibility:

\begin{thm}{\label{womaximal}}
 Let $\phi$ be a $\Pi_{2}$-statement in the language of set theory such that ZFC$\vdash\phi$. Then $\phi\leq_{\text{oW}}$WO.
\end{thm}
\begin{proof}
 Let $\phi=\forall{x}\exists{y}\psi(x,y)$, where $\psi$ is $\Delta_{0}$. We describe two programs $P$ and $Q$ such that $[Q,P]$ witnesses the oW-reducibility of $\phi$ to WO.
Let $F$ be an arbitrary canonification of WO.
$P$ works as follows: Given a code $c$ for a set $x$, compute a code $c^{\prime}$ for tc($x$), the transitive closure of $x$.
Then $F(c^{\prime})$ will be a code $c^{\prime\prime}$ for a well-ordering of tc($x$). Let $\preceq$ be that well-ordering.
Now $Q$ works as follows: Given $Y:=(\text{tc}(x),\preceq)$, a well-ordered set, $L[Y]$ will be a model of ZFC, so $L[Y]\models\exists{z}\psi(x,z)$ by assumption. 
Now $Q$ enumerates $L[Y]$ until some $L_{\alpha}[Y]$ is found that contains such a $z$ and returns the $<_{L[Y]}$-smallest such $z$.

Clearly then, $F_{Q}\circ F\circ F_{P}$ is a canonification of $\phi$.

\end{proof}

Finally, we note two further separation results (under $0^{\sharp}$):

\begin{thm}{\label{wo>ac}}[Cf. \cite{Ca16}, Thm. 10]
 AC$<_{\text{oW}}$WO
\end{thm}
\begin{proof}
As we already know AC$\leq_{\text{soW}}$WO, it suffices to show that 
WO$\nleq_{\text{oW}}$AC. This is shown in [\cite{Ca16}, Thm. 10]. The idea of the proof is that Zarach's construction of a model of ZF$^{-}+\text{AC}+\neg\text{WO}$ in \cite{Z} can (under the assumption of $0^{\sharp}$ be carried out over $L$ and yields a transitive class model $M$ of the same theory and contains a counterexample to WO that can be well-ordered by forcing over $M$ in two mutually generic ways, so that the assumptions of Lemma \ref{forcing method} are satisfied.

\end{proof}

The preceding result is plausible, as the proof of AC$\rightarrow$WO over ZF applies AC to the power set of the set to be well-ordered, and power sets are not computable by Corollary \ref{applications cardinality}. The following is thus to be expected. As usual, for $R,R^{\prime}\subseteq V^{2}$, we write 
$R\circ R^{\prime}$ for $\{(a,b):\exists{c}((a,c)\in R\wedge (c,b)\in R^{\prime})\}$; moreover, we write "Pot" for the relation $\{(x,\mathfrak{P}(x)):x\in V\}$.

\begin{prop}{\label{WO AC POT}}
WO$<_{\text{oW}}$AC$^{\prime}\circ$Pot. 
\end{prop}
\begin{proof}
$\leq_{\text{oW}}$ follows by the obvious effectivization of the proof of the implication AC$^{\prime}\rightarrow$WO over ZF. Since a choice function on the power set of a set $x$ has the same cardinality as that power set, strictness follows from Lemma \ref{cardinality method}.
\end{proof}

Note that we saw above that WO$\equiv_{\text{OTM}}$PP. Thus, we have seen that ordinal Weihrauch reducibility is strictly stronger than OTM-reducibility.


%

The following picture summarizes the situation; $<_{\text{oW}}$ is indicated by 
arrows, $\leq_{\text{oW}}$ by dotted arrows and $\equiv_{\text{OTM}}$ by a dashed arrow. All indicated oW-reducibilities (whether strict or not) are strong. 

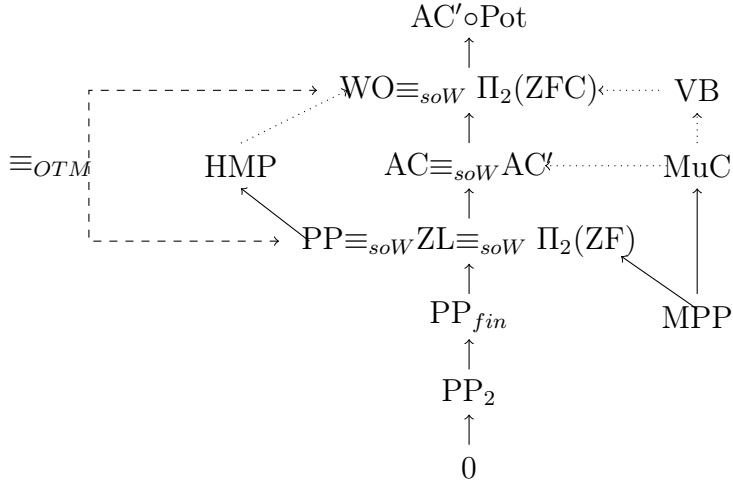
\begin{tikzpicture}
\node at (0,-7) {0};
\draw [->] (0,-6.7)--(0,-6.3);
\node at (0,-6) {PP$_{2}$};
\draw [->] (0,-5.7)--(0,-5.3);
\node at (0,-5) {PP$_{fin}$};
\draw [->] (0,-4.7)--(0,-4.3);
\node at (0,-4) {PP$\equiv_{soW}$ZL$\equiv_{soW}\Pi_{2}$(ZF)};
\draw [->] (0,-3.7)--(0,-3.3);
\node at (0,-3) {AC$\equiv_{soW}$AC$^{\prime}$};
\draw [->] (0,-2.7)--(0,-2.3);
\node at (0,-2) {WO$\equiv_{soW}\Pi_{2}$(ZFC)};
\draw [->] (0,-1.7)--(0,-1.3);
\node at (0,-1) {AC$^{\prime}\circ$Pot};

\draw [dashed, <->] (-2.5,-4)--(-5,-4)--(-5,-2)--(-2,-2);
\node at (-5.5,-3) {$\equiv_{OTM}$};

\draw [->] (3,-4.9)--(2,-4.2);
\node at (3,-5) {MPP};
\draw [->] (3,-4.7)--(3,-3.3);
\node at (3,-3) {MuC};
\draw [dotted, ->] (2.6,-3)--(1,-3); 
\draw [dotted,->] (3,-2.7)--(3,-2.3);
\node at (3,-2) {VB};
\draw [dotted,->] (2.5,-2)--(1.7,-2);

\node at (-3,-3) {HMP};
\draw [->] (-2.1,-4)--(-3,-3.3);
\draw [dotted, ->] (-3,-2.7)--(-1.6,-2);

\end{tikzpicture}

\bigskip
\textbf{Question}: We do not know the answers to the following questions:

\begin{enumerate}
\item Is AC$<_{\text{oW}}$HMP (or even AC$\leq_{\text{oW}}$HMP)?
\item Is HMP$<_{\text{oW}}$WO?
\item Is MuC$<_{\text{oW}}$AC?
\end{enumerate}

\section{Effectivity and Provability}

A question that has recently received attention in the classical theory of Weihrauch reducibility is whether the reducibility of a statement $\phi$ to another statement $\psi$ corresponds to 
the provability of the implication $\psi\rightarrow\phi$ in some logical calculus; partial answers to this have been obtained in Kypers \cite{Ku}. We expect a similar analysis to work in the transfinite setting. As a first observation, we show below that reducibility of a set-theoretical $\Pi_{2}$-sentence $\psi$ to $\phi\wedge\text{PP}$ 
for another such $\Pi_{2}$-sentence $\phi$ in the sense of $\leq_{\text{OTM}}$ (where arbitrarily many applications of a canonification are allowed) 
is implied by the provability of $\phi\rightarrow\psi$ in KP. (PP is necessary here to pick a particular witness after a non-empty set of witnesses has been determined at the end of the construction.) 
The result is certainly not optimal, so that one cannot expect the converse to hold; and in fact we have WO$\leq_{\text{OTM}}$PP, while the implication PP$\rightarrow$WO is not even provable in ZF, let alone in KP. 



\begin{lemma}{\label{reasonable witnesses}}
Let $\phi(x,y)$ be $\Delta_{0}$, and let $a$ be a set. If there is $b$ such that $\phi(a,b)$ holds in $V$, then there is $b^{\prime}$ such that $\phi(a,b^{\prime})$ holds in $V$ and such that 
$|\text{tc}(b^{\prime})|^{L[\text{tc}(a,b^{\prime})]}\leq|\text{tc}(a)|\aleph_{0}$.
\end{lemma}
\begin{proof}
Form the elementary hull of $\{\text{tc}(a)\}$ in level $L_{\alpha}[\text{tc}(a,b^{\prime})]$ that contains $b$ and use condensation.
\end{proof}

\begin{defini}
Let $\phi$ be $\Pi_{2}$. A canonification of $\phi$ is called `reasonable' if and only if $|\text{tc}(F(a))|\leq|\text{tc}(a)|\aleph_{0}$ for every set $a$.
\end{defini}

\begin{corollary}{\label{reasonable canonifications}}
Suppose that $F$ is a canonification of the $\Pi_{2}$-statement $\phi$. Then a reasonable canonification $G$ of $\phi$ is OTM-computable from $F$.
\end{corollary}
\begin{proof}
Effectivize the proof of Lemma \ref{reasonable witnesses}. That is, suppose $a$ is given; to simplify the argumentation, we will assume that $a$ is transitive, which we can assume without loss of generality by replacing $a$ with $t\text{tc}(a)$ if necessary. Use the miracle command to obtain $F(a)$. Then $L[a,\text{tc}(\{F(a)\})]$\footnote{We write $L[a,b]$ for $L[(a,b)]$.} will contain some $b$ with $\psi(a,b)$ (namely $F(a)$), and, by Lemma $22$, it will contain such a $b$ of the required cardinality. It thus suffices to enumerate $L[a,\text{tc}(\{F(a)\})]$ using the relativized version of $P_{L}$ and output the element that is minimal in the sense of the canonical well-ordering of $L[a,\text{tc}(F(a))]$.
\end{proof}

\begin{thm}{\label{kp prov red}}
Let $\phi,\psi\in\Pi_{2}$, and suppose that KP$\models\phi\rightarrow\psi$. Then $\psi\leq\phi\wedge\text{PP}$.
\end{thm}
\begin{proof}
Let $\phi$ be $\forall{x}\exists{y}\phi^{\prime}(x,y)$ and let $\psi$ be $\forall{x}\exists{y}\psi^{\prime}(x,y)$, where $\phi^{\prime}$ and $\psi^{\prime}$ are $\Delta_{0}$.
Moreover, let $F$ be a canonification for $\phi$ and let $a$ be an instance for $\psi$. By Corollary \ref{reasonable canonifications}, we may assume without loss of generality that $F$ is reasonable.

In order to compute, relative to $F$, a (code for a) set $b$ such that $\psi^{\prime}(a,b)$, we proceed as follows:

Given $a$, compute $(L_{\alpha}[a]:\alpha\in\text{On})$ up to the first $\alpha_{0}$ such that $L_{\alpha_{0}}[a]\models\text{KP}$. Let $M_{0}:=L_{\alpha_{0}}[a]$.

In general, let $\alpha_{\iota+1}$ be minimal such that $L_{\alpha_{\iota+1}}[F[M_{\alpha_{\iota}}]]\models\text{KP}$ and let $M_{\iota+1}:=L_{\alpha_{\iota+1}}[F[M_{\alpha_{\iota}}]]$.
For $\delta$ a limit ordinal, let $M_{\delta}:=\bigcup_{\iota<\delta}M_{\iota}$.

\bigskip

\textbf{Claim}: There is a limit ordinal $\beta$ such that $M_{\beta}\models\text{KP}+\phi$. 

\medskip

To see this, first note that all $M_{\beta}$ are transitive and increasing unions of admissible sets. Clearly, $M_{\delta}\models\phi$ whenever $\delta$ is a limit ordinal, as such sets 
are closed under $F$. Also, all axioms of KP expect $\Delta_{1}$-collection and -separation will clearly hold in such an $M_{\delta}$. It thus suffices to show that there is a limit ordinal $\delta$ such that $\Delta_{1}$-collection and $\Delta_{1}$-separation
holds in $M_{\delta}$. 

Let $\kappa$ be the smallest regular uncountable cardinal $>|a|$. We claim that $M_{\kappa}$ is as desired. Since $F$ is reasonable, it follows inductively that $|u|<\kappa$ for all $u\in M_{\kappa}$. 
To see that $\Sigma_{1}$-collection holds in $M_{\kappa}$, let $u,\vec{p}\in M_{\kappa}$, and let $H$ be a $\Sigma_1$-function over 
$M_{\kappa}$ such that $H$ is defined on $u$. We claim that there is $\gamma<\kappa$ such that $H$ is defined on $u$ in $M_{\gamma}$ (i.e. $M_{\gamma}$ contains the necessary witnesses); 
this suffices, as $M_{\gamma+1}\models\text{KP}$ will satisfy $\Sigma_{1}$-collection and thus contain the desired superset of $H[u]$, which will still work in $M_{\kappa}$ by upwards absoluteness of $\Sigma_{1}$-formulas. 
Now, if this failed, then the set of $\xi$ such that a new value of $H$ for some element of $u$ is defined in $M_{\xi}$ would be cofinal in $\kappa$ and have cardinality $<|u|$, which, by our initial observation, contradicts
the regularity of $\kappa$. Thus $\Sigma_{1}$-collection holds in $M_{\kappa}$. The same argument shows that $M_{\kappa}$ satisfies $\Sigma_{1}$-separation, and a fortiori $\Delta_{1}$-separation. This concludes the proof of the claim.

\medskip

Now, to compute a $b$ as desired, compute the sequence $(M_{\iota}:\iota\in\text{On})$ until a limit ordinal $\gamma$ is found for which $M_{\gamma}$ is admissible; that such a $\gamma$ exists follows from the claim, and thus it will eventually be found.
As $M_{\gamma}\models\text{KP}+\phi$ and $\text{KP}\models\phi\rightarrow\psi$, we have $M_{\gamma}\models\psi$. As $a\in M_{\gamma}$, there is $b\in M_{\gamma}$ such that $M_{\gamma}\models\psi^{\prime}(a,b)$, and thus, by absoluteness of 
$\Delta_{0}$, $\psi^{\prime}(a,b)$ will hold in $V$. But now, the set $B$ of all such $b$ from $M_{\gamma}$ can easily be determined by searching $M_{\gamma}$. By one application of PP to $B$, we finally obtain the required witness.
\end{proof}

\bigskip
\noindent
\textbf{Question}: Are there logical calculi $C$ and axiomatic systems $T$ such that reducibility of $\psi$ to $\phi$ in the sense of $\leq_{\text{OTM}}$ corresponds precisely to the provability of the implication $\phi\rightarrow\psi$ from $T$ in $C$? What about $\leq_{\text{oW}}$ or $\leq_{\text{soW}}$? It seems particularly tempting to look for connections to constructive set theory here.

\bigskip






We conclude by observing that the correspondence between provability and reducibility has its limits, at least as long as the underlying calculus is classical.

Let us say that an $\in$-theory $T$ effectivizes a binary relation $\mathcal{C}$ if every model $M$ of $T$ has, for every $x\in M$, some $y\in M$ such that $(x,y)\in\mathcal{C}$.

\begin{prop}
 Let $T$ be an OTM-computable $\in$-theory with transitive models in $L$. 
Then there is an $\in$-definable binary relation $\mathcal{C}$ that is effective, but for which $\forall{x}\exists{y}(x,y)\in\mathcal{C}$ is not provable in $T$.

Moreover, if the existence of transitive models of $T$ is consistent with ZFC, then so is the existence of such a $\mathcal{C}$.
\end{prop}
\begin{proof}
 Consider the relation that holds between each set $x$ and each transitive model of $T$ in $L$. Then $\mathcal{C}$ has an effective canonification: The program $P$ that enumerates $L$ until the first transitive model of $T$ is found and then outputs that model is of the desired kind. But $T$, being consistent, cannot prove that there are such models. 
\end{proof}

%

Consequently, for classical logic, the answer to our above question is negative. Given that connections between reducibility and provability have been based on intuitionistic or linear logic in the classical case, this is hardly surprising.

%

\section{Beyond $\Pi_{2}$}

The setting introduced above is restricted to $\Pi_{2}$-formulas. However, there are at least two ways to generalize canonifications, and thus reducibility notions, beyond this point, which we want to sketch briefly in this section. For technical convenience, let us assume from now on that all formulas are given in prenex normal form and let us consider the $\Pi_{n}$-statement 
$\phi\equiv\forall{x_{1}}\exists{y_{1}}\forall{x_{2}}\exists{y_{2}}...\forall{x_{n}}\exists{y{n}}\psi(x_{1},y_{1},x_{2},y_{2},...,x_{n},y_{n})$, where $\psi$ is $\Delta_{0}$. A straightforward generalization of the notion of a canonification given above Definition \ref{eff defini} would be to say that a canonification of $\phi$ is a function $F:V\rightarrow V$ such that, for all $a\in V$, we have $\forall{x_{2}}\exists{y_{2}}...\forall{x_{n}}\exists{y_{n}}\psi(a,F(a),x_{2},y_{2},...,x_{n},y_{n})$. Let us call this a "superficial" canonification, or $s$-canonification, of $\phi$. A more radical approach would be to demand the canonification to "penetrate" the whole quantifier nesting; so we define a "thorough" canonification, or $t$-canonification, to be an $n$-tuple $(F_{1},...,F_{n})$ of functions with $F_{i}:V^{i}\rightarrow V$ such that, for all $i\in\{1,2,...,n\}$ and all 
$a_{1},...,a_{i-1}\in V$, we have 
$\forall{x_{i}}\exists{y_{i}}...\forall{x_{n}}\exists{y_{n}}\psi(a_{1},F_{1}(a_{1}),a_{2},F_{2}(a_{1},a_{2}),...,a_{i-1},F_{i-1}(a_{1},...,a_{i-1}),x_{i},y_{i},...,x_{n},y_{n})$.

In the spirit of Definitions \ref{reducibility} and \ref{ordinal weihrauch reducibility} above, we can then introduce reducibility relations between $\Pi_{n}$-statements for arbitrary $n$. For $s$-canonicifations, the definitions carry over verbatim. For $t$-canonifications, let $\phi\in\Pi_{m}$, $\phi^{\prime}\in\Pi_{n}$.

\begin{itemize}
\item We say that $\phi$ is OTM$^{t}$-reducible to $\phi^{\prime}$ and write $\phi\leq_{\text{OTM}}^{t}\phi^{\prime}$ if and only if there is an OTM-program $P$ using $n$ "miracle" tapes $T_{1},...,T_{n}$ and $m$ output tapes $T_{1}^{\prime},...,T_{m}^{\prime}$, such that, whenever $(F_{1},...,F_{n})$ is a canonification of $\phi^{\prime}$ and carrying out the "miracle" command on $T_{i}$ replaces the $i$-tuple coded by the current content of $T_{i}$ by a code of its image under $F_{i}$, then $P$ computes the components of a canonification of $\phi$ on its output tapes.
\item We say that $\phi$ is ordinal Weihrauch$^{t}$-reducible to $\phi^{\prime}$ and write $\phi\leq^{t}_{\text{oW}}\phi^{\prime}$ if and only if there are there are functional OTM-programs $P$ and $Q$ with $F_{Q}:V^{m}\rightarrow V^{n}$ and $F_{P}:V^{n+1}\rightarrow V^{m}$ such that, for all $\vec{a}=(a_{1},...,a_{m})\in V$ and all canonication $(F_{1},...,F_{n})$ of $\phi^{\prime}$, if 
$F_{Q}(a_{1},...,a_{m})=(b_{1},...,b_{n})$ and $(c_{1},...,c_{m})=F_{P}(\vec{a},F_{1}(b_{1}),F_{2}(b_{1},b_{2}),...,F_{n}(b_{1},...,b_{n}))$, then 
the function tuple $(F_{1},...,F_{m})$ such that $F_{i}$ maps $(a_{1},...,a_{i})$ to $(c_{1},...,c_{i})$ for $1\leq i\leq m$ is a canonification of $\phi$. 
\end{itemize}

In order to be able to consider cases like the ZFC axioms, we can slightly further extend this to axiom schemes and say that $\{\phi_{i}:i\in\omega\}$ is ordinal Weihrauch
reducible to $\{\psi_{j}:j\in\omega\}$ if and only if we have a program $P$ that, for any $i\in\omega$, computes a $j\in\omega$ and further reduces $\phi_{i}$ to $\psi_{j}$.

By an easy adaptation of the proof, one obtains the following variant of Lemma \ref{cardinality method}. Let us say that $F:V^{n}\rightarrow V$ raises cardinalities if and only if there is a tuple $(a_{1},...,a_{n})$ such that $\mu:=\text{max}\{|\text{tc}(a_{1})|,...,|\text{tc}(a_{n})|\}$ is infinite and $|F(a_{1},...,a_{n})|>\mu$ and that a tuple $(F_{1},...,F_{n})$ of functions raises cardinalities if and only if any of its elements does; finally, let us say that a $\Pi_n$-statement $\phi$ raises cardinalities if and only if any canonification $(F_{1},...,F_{n})$ of $\phi$ raises cardinalities.

\begin{lemma}{\label{general cardinality method}}
Let $\phi\in\Pi_{n}$, $\psi\in\Pi_{m}$, and suppose that $\phi$ raises cardinalities, but $\psi$ does not. Then $\phi\nleq_{\text{oW}}^{t}\psi$.
\end{lemma}

We note some initial observations.

Let us denote by Pot the power set axiom, i.e. the statement that $\forall{x}\exists{y}\forall{z}(z\in y\leftrightarrow z\subseteq x)$.

\begin{corollary}{\label{pot on top}}
For no ZFC axiom or axiom scheme $\mathcal{A}$ other than Pot itself do we have Pot$\leq_{\text{oW}}^{t}\mathcal{A}$.
\end{corollary}

By an obvious effectivization of the well-known proof of the implication, one gets:

\begin{prop}
Let Sep denote the separation scheme and let Rep denote the replacement scheme. Then Sep$\leq_{\text{oW}}^{t}$Rep.
\end{prop}

\section{Conclusion and Further Work}

We have introduced notions of effectivity, reducibility and `case-wise' reducibility applicable to mathematical objects of arbitrary cardinality. The approach to effectivity is supported by the remarkable conceptual stability of ordinal computability (see e.g. \cite{Fi} or \cite{Ca}) and moreover, while not equivalent to e.g. the approach by Hodges, agrees with it concerning the results obtained so far. With regard to reducibility, we have seen how set-theoretical techniques can be used to distinguish between various versions of set-theoretical principles usually regarded as equivalent. 


Clearly, there is a host of questions asking which statements are effectively reducible or (s)oW-reducible to which others. This may be viewed as a cardinality-independent version of reverse mathematics (as e.g. considered in \cite{Sh}) and the theory of the Weihrauch lattice. Apart from that, it may be interesting to consider variants of these notions with parameter-free computability replaced by other models of transfinite computation, like Infinite Time Turing Machines (\cite{HL}) or OTMs with ordinal parameters. Another worthwhile topic would be to replace (relativized) computability with (relativized) recognizability (see e.g. \cite{CSW}). 

Although Hodges did not introduce reducibility notions in his \cite{Ho}, the notion of effectivity proposed there, based on primitive-recursive set functions, can easily be supplemented with such concepts following the example of Definitions \ref{reducibility} and \ref{ordinal weihrauch reducibility}. A potential advantage is that primitive recursive set functions apply immediately to sets with no need for coding, and thus the technical issues that in our approach need to be dealt with using forcing no longer present themselves. Moreover, it will be interesting to see whether reducibility and nonreducibility results like those obtained in the present paper are `stable' in the sense that they hold for various formalizations of reducibility in the transfinite.

Another topic from the classical setting that might yield to interesting results in the uncountable realm is that of `decomposability' of statements or relations; as e.g. in \cite{BGM}, we can say that the $\Pi_{2}$-statement $\phi$ is `oW-decomposable' if and only if 
there are $\Pi_{2}$-statements $\psi,\psi^{\prime}<_{\text{oW}}\phi$ such that $\phi$ is a $\leq_{\text{oW}}$-least upper bound for 
$\psi$ and $\psi^{\prime}$ in the $\leq_{\text{oW}}$-ordering. As a special case, we say that a $\Pi_{2}$-statement $\phi$ is `partitionable' if and only if there are disjoint OTM-decidable classes $X,Y\subseteq V$ such that $X\cup Y=V$ and such that both
$R_{0}:=\{(x,y):(x\in X\wedge (x,y)\in R_{\phi})\vee (x\notin X \wedge y=\emptyset)\}$ and
$R_{1}:=\{(y,z):(y\in Y\wedge (y,z)\in R_{\phi})\vee (y\notin Y \wedge z=\emptyset)\}$ are strictly oW-reducible to $R_{\phi}$.
A crucial step in studying partitionability of statements would be the verification of the following conjecture, which is currently open\footnote{See, however, the discussion on MathOverflow at \url{https://mathoverflow.net/questions/314053/delta1-2-and-degrees-of-constructibility-textbfon-sets}}:

\bigskip
\textbf{Conjecture}: Let $F:V\rightarrow\{0,1\}$ be OTM-computable. Then one of $F^{-1}[0]$ and $F^{-1}[1]$ contains sets of every degree of constructibility.

\bigskip

Relative to this conjecture, it is rather easy to obtain the following from the fact that the well-ordering principle holds in any $L[X]$ when $X$ is a well-ordered transitive set:

\bigskip
\textbf{Conjecture}: WO is not partitionable.

\bigskip

Finally, various notions from classical computability theory could be incorporated into our framework: For example, one should be able to make sense of the concept of a `random construction' and ask whether there are interesting non-effective constructions that are reducible to them. We will also consider candidates for a sensible notion of a `jump operator' for construction problems, a notion that led to a number of fascinating results about Weihrauch reducibility (\cite{BGM}).

\section{Acknowledgements}

First of all, we want to express our gratitude to the organizers of the Dagstuhl workshop 
"Measuring the Complexity of Computational Content: From Combinatorial Problems to Analysis" for providing 
an inspiring incentive to continue our work on this topic and writing this paper.

We thank the anonymous referees of our CiE 2016 paper \cite{Ca16} on which this paper is based for their help in improving the presentation. 

Finally, we thank De Gruyter for the permission to use some of the material appearing in chapter $8$ of the forthcoming \cite{Ca19} in this article.

%


\newpage

\end{document}